\documentclass[reqno,10pt]{amsart}
\usepackage[T1]{fontenc}
\usepackage{enumerate}
\theoremstyle{plain}    
\newtheorem{thm}{Theorem}

\newtheorem{lemma}{Lemma}

\begin{document}
\title{Three Trichotomy Theorems}
\author{ Frank J. Palladino }
\address{Department of Mathematics, University of Rhode Island, Kingston, RI 02881;}
\email{frank@math.uri.edu}
\date{April 7, 2013}

\begin{abstract}
\noindent We study the following $k^{th}$ order rational
difference equation assuming nonnegative parameters and nonnegative initial conditions
$$x_{n}=\frac{\alpha+\sum_{i=1}^{k}\beta_{i}x_{n-i}}
{A+\sum_{j=1}^{k}B_{j}x_{n-j}},\quad n\in\mathbb{N}.$$ 
We develop a new periodic trichotomy result which unifies all currently known periodic trichotomy results regarding the above difference equation into three major families. The periodic trichotomy result presented in this article also contains as special cases some new examples of periodic trichotomy behavior.\newline
\end{abstract}
\maketitle

\section{Introduction}
A periodic trichotomy is a type of periodic bifurcation which occurs for certain rational difference equations, characterized by a three way split of the qualitative behavior depending on the selection of the parameters. Generally, the behavior of such a periodic trichotomy can be described as follows. In a region of parametric space every solution converges to an equilibrium solution. On the boundary of that region, every solution converges to a periodic solution of not necessarily prime period $p$ and there exist periodic solutions of prime period $p$, with $p$ depending on the underlying circumstances. Outside of that region of parametric space and its boundary, an unbounded solution may be constructed with the appropriate choice of initial condition. 

Many authors have contributed periodic trichotomy results for special cases of the general linear fractional rational difference equation
\begin{equation}\label{eq:lf}
x_n=\frac{\alpha+\sum^{k}_{i=1}\beta_{i}x_{n-i}}{A+\sum^{k}_{j=1}B_{j}x_{n-j}},\quad n\in\mathbb{N},
\end{equation}
see for example, [1-15] and [18-21].
We cannot stress enough the importance of this preceding work on periodic trichotomies. The subsequent three theorems in this article can be seen as the culmination of a long line of preceding work outlined in the previous citations.

All currently known periodic trichotomy results for Equation \eqref{eq:lf} can be organized into three major families described by the following three theorems. In this article, we give a streamlined and relatively short proof of these three theorems. To state these theorems we need the following notation $$I_{\beta}=\{i\in\{1,\dots, k\}|\beta_{i}>0\}\mbox{ and }I_{B}=\{j\in\{1,\dots, k\}|B_{j}>0\}.$$

\begin{thm} \label{thm:trichotomy2}
Consider the $k^{th}$ order rational difference equation,
\begin{equation} \label{eq:trichotomy1}
x_n=\frac{\sum^{k}_{i=1}\beta_{i}x_{n-i}}{A+\sum^{k}_{j=1}B_{j}x_{n-j}},\quad n\in\mathbb{N}.
\end{equation}
Assume nonnegative parameters and nonnegative initial conditions. Further
assume that $\sum^{k}_{i=1}\beta_{i}>0$, $A>0$, and that there does not exist
$j\in I_B$ so that $gcd(I_\beta) | j$. Under these
assumptions solutions of Equation \eqref{eq:trichotomy1} exhibit the following trichotomy behavior.
\begin{enumerate}[i.]
\item When $A>\sum^{k}_{i=1}\beta_{i}$, the unique equilibrium is globally asymptotically stable.
\item When $A=\sum^{k}_{i=1}\beta_{i}$, every solution converges to a periodic
solution of not necessarily prime period $gcd(I_{\beta})$, and there exist periodic solutions of prime period $gcd(I_{\beta})$.
\item When $A < \sum^{k}_{i=1}\beta_{i}$, unbounded solutions exist for some choice of initial
conditions.
\end{enumerate}
\end{thm}

\begin{thm} \label{thm:trichotomy1}
Consider the $k^{th}$ order rational difference equation,
\begin{equation} \label{eq:trichotomy2}
x_n=\frac{\alpha+\sum^{k}_{i=1}\beta_{i}x_{n-i}}{A+\sum^{k}_{j=1}B_{j}x_{n-j}},\quad n\in\mathbb{N}.
\end{equation}
Assume nonnegative parameters and nonnegative initial conditions so that the denominator is nonvanishing. Further
assume that $\sum^{k}_{j=1}B_{j}>0$ and $\alpha>0$. Moreover assume $2gcd(I_\beta \cup I_B)|i$ for all $i\in I_{\beta}$ and $2gcd(I_\beta \cup I_B)|(j+gcd(I_\beta \cup I_B))$ for all $j\in I_B$. Under these
assumptions Equation \eqref{eq:trichotomy2} exhibits the following trichotomy behavior.
\begin{enumerate}[i.]
\item When $A>\sum^{k}_{i=1}\beta_{i}$, the unique equilibrium is globally asymptotically stable.
\item When $A=\sum^{k}_{i=1}\beta_{i}$, every solution converges to a periodic solution of not necessarily prime period $2gcd(I_{\beta}\cup I_{B})$, and there exist periodic solutions of prime period $2gcd(I_{\beta}\cup I_{B})$.
\item When $A < \sum^{k}_{i=1}\beta_{i}$, unbounded solutions exist for some choice of initial
conditions.
\end{enumerate}
\end{thm}

\begin{thm} \label{thm:trichotomy3}
Consider the $k^{th}$ order rational difference equation,
\begin{equation} \label{eq:trichotomy3}
x_n=\frac{\alpha+\sum^{k}_{i=1}\beta_{2i}x_{n-2i}+x_{n-\ell}}{A+x_{n-\ell}},\quad n\in\mathbb{N}.
\end{equation}
Assume nonnegative parameters, positive initial conditions and that $\ell$ is odd. Under these
assumptions Equation \eqref{eq:trichotomy3} exhibits the following trichotomy behavior.
\begin{enumerate}[i.]
\item When $A+1>\sum^{k}_{i=1}\beta_{2i}$, every solution converges to an equilibrium.
\item When $A+1=\sum^{k}_{i=1}\beta_{2i}$, every solution converges to a periodic solution of not necessarily prime period $2gcd(I_{\beta})$, and there exist periodic solutions of prime period $2gcd(I_{\beta})$.
\item When $A+1 < \sum^{k}_{i=1}\beta_{2i}$, unbounded solutions exist for some choice of initial
conditions.
\end{enumerate}
\end{thm}  

Theorems \ref{thm:trichotomy2} and \ref{thm:trichotomy1} first appeared in \cite{fptri}. Theorem \ref{thm:trichotomy3} is new to this article.
 
\section{Details}

In this section we will prove Theorems 1-3. We begin with the following lemmas.
\begin{lemma}\label{l:ineq}
For nonnegative $a,b$ and positive $c,d$,
$$\min\left(\frac{a}{c},\frac{b}{d}\right)\leq\frac{a+b}{c+d}\leq\max\left(\frac{a}{c},\frac{b}{d}\right).$$
\end{lemma}
\begin{proof}
Multiplying through by $c+d$ yields,
$$\min\left(a+\frac{ad}{c},b+\frac{bc}{d}\right)\leq a+b\leq\max\left(a+\frac{ad}{c},b+\frac{bc}{d}\right).$$
Now either $ad\geq bc$ or $ad\leq bc$, in both cases the above string of inequalies is true.
\end{proof}

\begin{lemma}\label{lemma:converge}
Consider the $k^{th}$ order rational difference equation,
$$x_n=\frac{\alpha +\sum^{k}_{i=1}\beta_{i}x_{n-i}}{A+\sum^{k}_{j=1}B_{j}x_{n-j}},\quad n\in\mathbb{N}.$$
Assume nonnegative parameters, nonnegative initial conditions, and
assume that $A> \sum^{k}_{i=1}\beta_{i}$, then every solution converges to an equilibrium.
\end{lemma}
\begin{proof}
We have $$x_{n}\leq \frac{\alpha}{A}+\frac{\sum_{i\in I_{\beta}}\beta_{i}x_{n-i}}{A}\leq \frac{\alpha}{A}+\left(\frac{\sum_{i\in I_{\beta}}\beta_{i}}{A}\right) max_{i\in I_{\beta}}x_{n-i},\;\;n\in\mathbb{N}.$$ This implies that every solution is bounded by Theorem 3 in \cite{f1}. Let $S=\limsup x_{n}$ and $I=\liminf x_{n}$. Then we have, $$S\leq \frac{\alpha + \sum_{i\in I_{\beta}}\beta_{i}S}{A+\sum_{j\in I_{B}}B_{j}I}\mbox{ and }I\geq \frac{\alpha + \sum_{i\in I_{\beta}}\beta_{i}I}{A+\sum_{j\in I_{B}}B_{j}S}.$$ So $0\leq \left(\sum_{i\in I_{\beta}}\beta_{i}-A\right)(S-I)$, which forces $S=I$.
\end{proof}

\begin{lemma}\label{lemma:trichotomy1}

Consider the $k^{th}$ order rational difference equation,
$$x_n=\frac{\sum^{k}_{i=1}\beta_{i}x_{n-i}}{A+\sum^{k}_{j=1}B_{j}x_{n-j}},\quad n\in\mathbb{N}.$$
Assume nonnegative parameters, nonnegative initial conditions, and
assume that $A\geq \sum^{k}_{i=1}\beta_{i}>0$, then every solution converges to a periodic solution of not necessarily prime period $gcd(I_{\beta})$.
\end{lemma} 
\begin{proof}
For notational purposes let $\rho=\left\lfloor\frac{k}{gcd(I_{\beta})}\right\rfloor$. $$\mbox{Put }y^{a}_{m}=\max_{\ell=1,\dots, \rho}(x_{(m-\ell)gcd(I_{\beta})+a}).$$
Notice that
$$x_{(m-1)gcd(I_{\beta})+a}= \frac{\sum^{k}_{i=1}\beta_{i}x_{(m-1)gcd(I_{\beta})+a-i}}{A+\sum^{k}_{j=1}B_{j}x_{(m-1)gcd(I_{\beta})+a-j}}\leq$$
$$ \frac{\sum^{k}_{i=1}\beta_{i}x_{(m-1)gcd(I_{\beta})+a-i}}{A}\leq \max_{i\in I_{\beta}}(x_{(m-1)gcd(I_{\beta})+a-i})\leq y^{a}_{m-1}.$$
$$\mbox{So }y^{a}_{m}= \max_{\ell=1,\dots, \rho}(x_{(m-\ell)gcd(I_{\beta})+a}) = \max(x_{(m-1)gcd(I_{\beta})+a}, \max_{\ell=2,\dots, \rho}(x_{(m-\ell)gcd(I_{\beta})+a}))\leq y^{a}_{m-1}.$$
So $\{y^{a}_{m}\}^{\infty}_{m=1}$ is monotone decreasing and bounded below by zero for each $a$, thus $\{y^{a}_{m}\}^{\infty}_{m=1}$ converges for each $a$ to a limit, which we will call $y^{a}_{*}$. Now, we claim that each subsequence $\{x_{mgcd(I_{\beta})+a}\}^{\infty}_{m=1}$ must also converge to $y^{a}_{*}$. The definition of $y^{a}_{m}$  tells us that $x_{mgcd(I_{\beta})+a}\leq y^{a}_{m+1}$. Thus $\limsup x_{mgcd(I_{\beta})+a}\leq \limsup y^{a}_{m}= y^{a}_{*}$. Now suppose for the sake of contradiction that $\liminf x_{mgcd(I_{\beta})+a}= I_{a}< y^{a}_{*}$. Then there is a further subsequence $\{x_{m_{b}gcd(I_{\beta})+a}\}$ which converges to $I_{a}$. This implies that $\limsup x_{m_{b}gcd(I_{\beta})+a+\eta} < y^{a}_{*}$ for any $\eta\in \{\sum^{\nu}_{m=1}i_{m}|\nu\in\mathbb{N}$ and $i_{m}\in I_{\beta}$ for all
$m\in \mathbb{N}\}$. We prove this via induction on $\nu$. In the base case $\nu=1$, $\eta\in I_{\beta}$ and so 
$$\limsup x_{m_{b}gcd(I_{\beta})+a+\eta}\leq \frac{\sum^{k}_{i=1}\beta_{i}\limsup x_{m_{b}gcd(I_{\beta})+a+\eta-i}}{A} \leq y^{a}_{*} + \frac{\beta_{\eta}(I_{a}-y^{a}_{*})}{A}.$$ Assume that the result is true for all $\nu < N$. Then take $\eta= \sum^{N}_{m=1}i_{m}= i_{N}+ \sum^{\nu}_{m=1}i_{m}$ for some $i_{N}\in I_{\beta}$. So $$\limsup x_{m_{b}gcd(I_{\beta})+a+\eta}\leq \frac{\sum^{k}_{i=1}\beta_{i}\limsup x_{m_{b}gcd(I_{\beta})+a+\eta-i}}{A} \leq $$
$$y^{a}_{*} + \frac{\beta_{i_{N}}(\limsup x_{m_{b}gcd(I_{\beta})+a+\sum^{\nu}_{m=1}i_{m}}-y^{a}_{*})}{A}<y^{a}_{*}.$$ Let $N_{f}$ be the Frobenius number of the set $\{\frac{i}{gcd(I_{\beta})}|i\in I_{\beta}\}$, then $\limsup y^{a}_{m_{b}+N_{f}+k} < y^{a}_{*}$ by the properties of the Frobenius number which is a contradiction.
\end{proof}

\begin{lemma}\label{lemma:bounds}
If $x_{n}\geq \min_{i=1,\dots, k}(x_{n-i},c)$, where $c>0$. Then $x_{n}$ is bounded below by $\min_{j=1,\dots, k}(x_{N-j},c)$ for $n\geq N$. Moreover if $x_{n}\leq \max_{i=1,\dots, k}(x_{n-i},c)$, where $c>0$. Then $x_{n}$ is bounded above by $\max_{j=1,\dots, k}(x_{N-j},c)$ for $n\geq N$.
\end{lemma}
\begin{proof}
We will prove the first case, the second case follows similarly. We prove this via strong induction on $n$, the case $n=N$ provides the base case. Assume the result is true for $N\leq n < J$. Then $$x_{J}\geq \min_{i=1,\dots, k}(x_{J-i},c)\geq \min(\min_{J-N< i \leq k}(x_{J-i}),\min_{j=1,\dots, k}(x_{N-j},c),c)$$$$\geq \min(\min_{ 1\leq \rho \leq k-J+N}(x_{N-\rho}),\min_{j=1,\dots, k}(x_{N-j},c),c)\geq \min_{j=1,\dots, k}(x_{N-j},c).$$
The first inequality comes from the original recursive inequality in the statement of the lemma. The second inequality comes from the induction hypothesis. Indeed if $i\leq J-N$, then $N\leq J-i < J$ and so $x_{J-i}\geq \min_{j=1,\dots, k}(x_{N-j},c)$. The third inequality comes from the fact that if $i>J-N$, then we may write $J-i=N-\rho$ where $\rho=i- (J-N)$. 
\end{proof}

\begin{lemma}\label{lemma:trichotomy2}
Consider the $k^{th}$ order rational difference equation,
$$x_n=\frac{\sum^{k}_{i=1}\beta_{i}x_{n-i}}{A+\sum^{k}_{j=1}B_{j}x_{n-j}},\quad n\in\mathbb{N}.$$
Assume nonnegative parameters, nonnegative initial conditions, $A>0$ and $gcd(I_{\beta})$ does not divide $j$ for any $j\in I_{B}$. Choose initial conditions $x_{-m}$ so that $x_{-m}=0$ for all $-m \not \equiv 0 \mod \gcd(I_{\beta})$. Under this choice of initial conditions $$x_n=\frac{\sum^{k}_{i=1}\beta_{i}x_{n-i}}{A},\quad n\in\mathbb{N}.$$
\end{lemma}
\begin{proof}
Using the initial conditions as the base case we may prove by induction that $x_{n}=0$ for all $n\not\equiv 0 \mod \gcd(I_{\beta})$. Suppose that the statement is true for all $n<N$. If $N\equiv 0 \mod \gcd(I_{\beta})$, then the statement is true for $N$ vacuously. If $N\not\equiv 0 \mod \gcd(I_{\beta})$, then $x_{N}= \frac{\sum^{k}_{i=1}\beta_{i}x_{N-i}}{A+\sum^{k}_{j=1}B_{j}x_{N-j}}=0$, since $N-i\not\equiv 0 \mod \gcd(I_{\beta})$ for all $i\in I_{\beta}$. $$\mbox{So clearly, } x_n=\frac{\sum^{k}_{i=1}\beta_{i}x_{n-i}}{A},\quad n\not\equiv 0 \mod \gcd(I_{\beta}).$$
Suppose $n\equiv 0 \mod \gcd(I_{\beta})$, since $gcd(I_{\beta})$ does not divide $j$ for any $j\in I_{B}$, $n-j\not\equiv 0 \mod \gcd(I_{\beta})$ for all $j\in I_{B}$. $$\mbox{So, } x_n=\frac{\sum^{k}_{i=1}\beta_{i}x_{n-i}}{A+\sum^{k}_{j=1}B_{j}x_{n-j}}=\frac{\sum^{k}_{i=1}\beta_{i}x_{n-i}}{A},\quad n\equiv 0 \mod \gcd(I_{\beta}).$$
\end{proof}

We now prove Theorem \ref{thm:trichotomy2}.
\begin{proof}
Case i. folllows from Lemma \ref{lemma:converge}.
Case ii.\newline
From Lemma \ref{lemma:trichotomy1} we see that every solution converges to a periodic solution of not necessarily prime period $\gcd(I_{\beta})$. From Lemma \ref{lemma:trichotomy2} we see that if we choose initial conditions so that $x_{-m}=1$ for all $-m \equiv 0 \mod \gcd(I_{\beta})$ and $x_{-m}=0$ for all $-m \not \equiv 0 \mod \gcd(I_{\beta})$, then this choice of initial conditions results in a periodic solution of prime period $\gcd(I_{\beta})$.\newline
Case iii.\newline
From Lemma \ref{lemma:trichotomy2} we see that if we choose initial conditions so that $x_{-m}=1$ for all $-m \equiv 0 \mod \gcd(I_{\beta})$ and $x_{-m}=0$ for all $-m \not \equiv 0 \mod \gcd(I_{\beta})$, then $$x_n=\frac{\sum^{k}_{i=1}\beta_{i}x_{n-i}}{A},\quad n\in\mathbb{N}.$$
 So, in this case with this choice of initial conditions our result follows quickly from results on recursive linear equations.\end{proof}

We now prove Theorem \ref{thm:trichotomy1}.
\begin{proof}
Case i. folllows from Lemma \ref{lemma:converge}. Case ii. \newline For notational purposes let $g=gcd(I_{\beta}\cup I_{B})$ and let $\rho = \left\lfloor \frac{k}{2g}\right\rfloor$. Put $$y^{a}_{m}=max_{\ell=1,\dots,\rho}\left(x_{2g(m-\ell)+a}, \frac{\alpha}{x_{2g(m-\ell)+a - g}\sum_{j\in I_{B}}B_{j}},\frac{\alpha}{x_{2gm+a-g}\sum_{j\in I_{B}}B_{j}}\right).$$ Then we have,
$$x_{2gm+a}\leq \frac{\alpha + \sum_{i\in I_{\beta}}\beta_{i}y^{a}_{m}}{A+\sum_{j\in I_{B}}B_{j}\frac{\alpha}{y^{a}_{m}\sum_{j\in I_{B}}B_{j}}}= y^{a}_{m},$$
since $A= \sum^{k}_{i=1}\beta_{i}$. Also, we have
$$x_{2gm+a+g}\geq \frac{\alpha + \sum_{i\in I_{\beta}}\beta_{i}\frac{\alpha}{y^{a}_{m}\sum_{j\in I_{B}}B_{j}}}{A+\sum_{j\in I_{B}}B_{j}y^{a}_{m}}=\frac{\alpha}{y^{a}_{m}\sum_{j\in I_{B}}B_{j}},$$
since $A=\sum^{k}_{i=1}\beta_{i}$. Thus $y^{a}_{m+1}\leq y^{a}_{m}$ for all $m,a\in\mathbb{N}$. So $\{y^{a}_{m}\}^{\infty}_{m=1}$ is monotone decreasing and bounded below by zero for each $a$, thus $\{y^{a}_{m}\}^{\infty}_{m=1}$ converges for each $a$ to a limit, which we will call $y^{a}_{*}$. Now, we claim that each subsequence $\{x_{2gm+a}\}^{\infty}_{m=1}$ must also converge to $y^{a}_{*}$. The definition of $y^{a}_{m}$  tells us that $x_{2gm+a}\leq y^{a}_{m+1}$. Thus $\limsup x_{2gm+a}\leq \limsup y^{a}_{m}= y^{a}_{*}$. Now suppose for the sake of contradiction that $\liminf x_{2gm+a}= I_{a}< y^{a}_{*}$. Then there is a further subsequence $\{x_{2gm_{b}+a}\}$ which converges to $I_{a}$. This implies that $\limsup x_{2gm_{b}+a} < y^{a}_{*}$. Now we need to make use of three recursive inequalities. Suppose $\limsup x_{2gm_{b}+a+c}<y^{a}_{*}$ then for any $\theta\in I_{B}$,
$$\liminf x_{2gm_{b}+a+c+\theta}\geq \frac{\alpha+\sum^{k}_{i=1}\beta_{i}\frac{\alpha}{y^{a}_{*}\sum_{j\in I_{B}}B_{j}}}{A+\sum_{j\in I_{B}}B_{j}\limsup x_{2gm_{b}+a+c+\theta-j}}> \frac{\alpha}{y^{a}_{*}\sum_{j\in I_{B}}B_{j}}.$$ Suppose $\liminf x_{2gm_{b}+a+c}>\frac{\alpha}{y^{a}_{*}\sum_{j\in I_{B}}B_{j}}$ then for any $\theta\in I_{B}$,
$$\limsup x_{2gm_{b}+a+c+\theta}\leq \frac{\alpha+\sum^{k}_{i=1}\beta_{i}y^{a}_{*}}{A+\sum_{j\in I_{B}}B_{j}\liminf x_{2gm_{b}+a+c+\theta-j}}< y^{a}_{*}.$$ Moreover, assume $\limsup x_{2gm_{b}+a+c}<y^{a}_{*}$ then for any $\eta\in I_{\beta}$,
$$\limsup x_{2gm_{b}+a+\eta}\leq \frac{\alpha + \sum^{k}_{i=1}\beta_{i}\limsup x_{2gm_{b}+a+\eta-i}}{A+\sum_{j\in I_{B}}B_{j}\frac{\alpha}{y^{a}_{*}\sum_{j\in I_{B}}B_{j}}} \leq y^{a}_{*} + \frac{\beta_{\eta}(I_{a}-y^{a}_{*})}{A+\sum_{j\in I_{B}}B_{j}\frac{\alpha}{y^{a}_{*}\sum_{j\in I_{B}}B_{j}}}< y^{a}_{*}.$$ So, using those three facts inductively we get $\limsup x_{2gm_{b}+a+\eta} < y^{a}_{*}$ for any $\eta\in \{\sum^{\nu}_{m=1}i_{m}+\sum^{\mu}_{m=1}j_{m}|\nu,\mu\in\mathbb{N}$, $\mu$ is even, $i_{m}\in I_{\beta}$, and $j_{m}\in I_{B}\}$. Moreover, $\limsup \frac{\alpha}{x_{2gm_{b}+a+\theta}\sum_{j\in I_{B}}B_{j}} < y^{a}_{*}$ for any $\theta\in \{\sum^{\nu}_{m=1}i_{m}+\sum^{\mu}_{m=1}j_{m}|\nu,\mu\in\mathbb{N}$, $\mu$ is odd, $i_{m}\in I_{\beta}$, and $j_{m}\in I_{B}\}$. Let $N_{f}$ be the Frobenius number of the set $\{\frac{i}{2g}|i\in I_{\beta}\}\cup\{\frac{i+j}{2g}|i,j\in I_{B}\}$, then $\limsup y^{a}_{m_{b}+N_{f}+k} < y^{a}_{*}$ by the properties of the Frobenius number which is a contradiction. We have just shown that each subsequence $\{x_{2gm+a}\}^{\infty}_{m=1}$ converges. Thus every solution converges to a periodic solution of not necessarily prime period  $2gcd(I_{\beta}\cup I_{B})$. Now, let us construct a periodic solution of prime period $2gcd(I_{\beta}\cup I_{B})$. Choose initial conditions so that $x_{-m}=\frac{\overline{x}}{2}$ for $-m\equiv 0 \mod 2gcd(I_{\beta}\cup I_{B})$, $x_{-m}=\frac{2\alpha}{\overline{x}\sum_{j\in I_{B}}B_{j}}$ for $-m\equiv gcd(I_{\beta}\cup I_{B}) \mod 2gcd(I_{\beta}\cup I_{B})$, and $x_{-m}=\overline{x}$ otherwise. These initial conditions give a periodic solution of prime period $2gcd(I_{\beta}\cup I_{B})$. We will prove this via strong induction on $n$ with the initial conditions providing the base case. Assume that $x_{n}=\frac{\overline{x}}{2}$ for $n\equiv 0 \mod 2gcd(I_{\beta}\cup I_{B})$, $x_{n}=\frac{2\alpha}{\overline{x}\sum_{j\in I_{B}}B_{j}}$ for $n\equiv gcd(I_{\beta}\cup I_{B}) \mod 2gcd(I_{\beta}\cup I_{B})$, and $x_{n}=\overline{x}$ otherwise for $n<N$. If $N\equiv 0 \mod 2gcd(I_{\beta}\cup I_{B})$, then since $2gcd(I_\beta \cup I_B)|i$ for all $i\in I_{\beta}$, $N-i\equiv 0 \mod 2gcd(I_{\beta}\cup I_{B})$ for all $i\in I_{\beta}$. Moreover, since $2gcd(I_\beta \cup I_B)|(j+gcd(I_\beta \cup I_B))$ for all $j\in I_B$, $N-j\equiv gcd(I_{\beta}\cup I_{B}) \mod 2gcd(I_{\beta}\cup I_{B})$ for all $j\in I_{B}$. So we have,
$$x_{N}=\frac{\alpha+\sum^{k}_{i=1}\beta_{i}x_{N-i}}{A+\sum^{k}_{j=1}B_{j}x_{N-j}}= \frac{\alpha+\sum^{k}_{i=1}\beta_{i}\frac{\overline{x}}{2}}{A+\sum^{k}_{j=1}B_{j}\frac{2\alpha}{\overline{x}\sum_{j\in I_{B}}B_{j}}}= \frac{2\alpha+\sum^{k}_{i=1}\beta_{i}\overline{x}}{2A+\frac{4\alpha}{\overline{x}}}=\frac{\overline{x}}{2},$$
since $\sum^{k}_{i=1}\beta_{i}=A$. If $N\equiv gcd(I_{\beta}\cup I_{B}) \mod 2gcd(I_{\beta}\cup I_{B})$, then since $2gcd(I_\beta \cup I_B)|i$ for all $i\in I_{\beta}$, $N-i\equiv gcd(I_{\beta}\cup I_{B}) \mod 2gcd(I_{\beta}\cup I_{B})$ for all $i\in I_{\beta}$. Moreover, since $2gcd(I_\beta \cup I_B)|(j+gcd(I_\beta \cup I_B))$ for all $j\in I_B$, $N-j\equiv 0 \mod 2gcd(I_{\beta}\cup I_{B})$ for all $j\in I_{B}$. So we have,
$$x_{N}=\frac{\alpha+\sum^{k}_{i=1}\beta_{i}x_{N-i}}{A+\sum^{k}_{j=1}B_{j}x_{N-j}}= \frac{\alpha+\sum^{k}_{i=1}\beta_{i}\frac{2\alpha}{\overline{x}\sum_{j\in I_{B}}B_{j}}}{A+\sum^{k}_{j=1}B_{j}\frac{\overline{x}}{2}}= \frac{2\alpha}{\overline{x}\sum_{j\in I_{B}}B_{j}},$$
since $\sum^{k}_{i=1}\beta_{i}=A$. If $N\not\equiv 0,gcd(I_{\beta}\cup I_{B}) \mod 2gcd(I_{\beta}\cup I_{B})$, then $N-i\not\equiv 0,gcd(I_{\beta}\cup I_{B}) \mod 2gcd(I_{\beta}\cup I_{B})$ for $i\in I_{\beta}\cup I_{B}$. Thus ,
$$x_{N}=\frac{\alpha+\sum^{k}_{i=1}\beta_{i}x_{N-i}}{A+\sum^{k}_{j=1}B_{j}x_{N-j}}=\frac{\alpha+\sum^{k}_{i=1}\beta_{i}\overline{x}}{A+\sum^{k}_{j=1}B_{j}\overline{x}}=\overline{x}.$$ Thus, we have demonstrated via induction that our choice of initial conditions gives a periodic solution of prime period $2gcd(I_{\beta}\cup I_{B})$. 
\newline Case iii. follows immediately from Theorem 1 of \cite{lp}.
\end{proof}

We now prove Theorem \ref{thm:trichotomy3}.
\begin{proof}
Suppose $\sum^{k}_{i=1}\beta_{2i}=0$, then the difference equation decouples into $\ell$ Riccati equations and the result quickly follows. So we may assume for the remainder of the proof that $\sum^{k}_{i=1}\beta_{2i}>0$.
If $\alpha+\sum^{k}_{i=1}\beta_{2i}\geq A$, then
$$x_{n}-1=\frac{\alpha-A+\sum^{k}_{i=1}\beta_{2i}x_{n-2i}}{A+x_{n-\ell}}=\frac{\sum^{k}_{i=1}\beta_{2i}+\alpha-A+\sum^{k}_{i=1}\beta_{2i}(x_{n-2i}-1)}{A+1+(x_{n-\ell}-1)}\;\; n\in\mathbb{N}.$$
So letting $w_{n}=x_{n}-1$,
\begin{equation}\label{eq:change1}
w_{n}=\frac{\sum^{k}_{i=1}\beta_{2i}+\alpha-A+\sum^{k}_{i=1}\beta_{2i}w_{n-2i}}{A+1+w_{n-\ell}}\;\; n\in\mathbb{N}.
\end{equation}
In the case $\alpha\geq A$ we have
$$x_{n}-1=\frac{\alpha-A+\sum^{k}_{i=1}\beta_{2i}x_{n-2i}}{A+x_{n-\ell}}\geq 0\;\; n\in\mathbb{N},$$ 
and the result follows immediately from Theorem \ref{thm:trichotomy1} after the change of variables. If $0<\alpha < A$, then let $r$ be the positive root of the equation
$$h(t)=t^{2}+(\sum^{k}_{i=1}\beta_{2i}+1-A)t-\alpha.$$
We have,
$$h(A)=A^{2}+(\sum^{k}_{i=1}\beta_{2i}+1-A)A-\alpha= \sum^{k}_{i=1}\beta_{2i}A+A-\alpha > 0\mbox{ and }h(0)=-\alpha<0,$$
so $r<A$. Thus, put $w_{n}=\frac{x_{n}+r}{1+r}$ and
$$w_{n}=\frac{\alpha+Ar+(r+1)x_{n-\ell}+\sum^{k}_{i=1}\beta_{2i}x_{n-2i}}{(1+r)(A+x_{n-\ell})}=\frac{\frac{\alpha+Ar-r^{2}-r-\sum^{k}_{i=1}\beta_{2i}r}{1+r}+(r+1)w_{n-\ell}+\sum^{k}_{i=1}\beta_{2i}w_{n-2i}}{A-r+(1+r)w_{n-\ell}}$$
$$=\frac{(r+1)w_{n-\ell}+\sum^{k}_{i=1}\beta_{2i}w_{n-2i}}{A-r+(1+r)w_{n-\ell}}\;\;n\in\mathbb{N}.$$
Thus in the case $0<\alpha<A$ the result follows immediately from the case $\alpha=0$ and in the case $\alpha\geq A$ the result follows immediately from Theorem \ref{thm:trichotomy1}. So we may assume without loss of generality that $A>0$ and $\alpha=0$ and we need only study the difference equation of the form 
\begin{equation}\label{eq:red}
x_n=\frac{\sum^{k}_{i=1}\beta_{2i}x_{n-2i}+x_{n-\ell}}{A+x_{n-\ell}},\quad n\in\mathbb{N}.
\end{equation}
  Suppose $0<\sum^{k}_{i=1}\beta_{2i}<A$, then by Lemma \ref{l:ineq},
$$x_{n}=\frac{\sum^{k}_{i=1}\beta_{2i}x_{n-2i}+x_{n-\ell}}{A+x_{n-\ell}}\leq \max\left(\frac{\sum^{k}_{i=1}\beta_{2i}x_{n-2i}}{A},1\right).$$
So the solution is bounded above by Lemma \ref{lemma:bounds}. Let $S=\limsup x_{n}$ and $I=\liminf x_{n}$. Then, $S\leq max(\frac{\sum^{k}_{i=1}\beta_{2i}}{A}S,1)$, which forces $S\leq \frac{A}{\sum^{k}_{i=1}\beta_{2i}}$. So the interval $[0,1]$ is an invariant attracting interval. On this interval the difference equation \eqref{eq:red} is increasing in all arguments. So we have,
$$I\geq\frac{(\sum^{k}_{i=1}\beta_{2i}+1)I}{A+I}\mbox{  and  }S\leq\frac{(\sum^{k}_{i=1}\beta_{2i}+1)S}{A+S}\mbox{. So,}$$
$$I^{2}\geq (\sum^{k}_{i=1}\beta_{2i}+1-A)I\mbox{  and  }S^{2}\leq (\sum^{k}_{i=1}\beta_{2i}+1-A)S.$$
Now, if $A\geq 1+\sum^{k}_{i=1}\beta_{2i}$, then this forces $S=0=I$. On the other hand, if $A< 1+\sum^{k}_{i=1}\beta_{2i}$, then put $\delta=A-\sum^{k}_{i=1}\beta_{2i}+\frac{1+\sum^{k}_{i=1}\beta_{2i}-A}{2}$. We may write, by Lemma \ref{l:ineq}, 
$$x_{n}=\frac{\sum^{k}_{i=1}\beta_{2i}x_{n-2i}+\delta x_{n-\ell}+ (1-\delta)x_{n-\ell}}{A+x_{n-\ell}}\geq \min\left(\frac{\sum^{k}_{i=1}\beta_{2i}x_{n-2i}+\delta x_{n-\ell}}{A},1-\delta\right).$$
Thus, since we have assumed positive initial conditions, the solution is bounded below by some $L>0$ by Lemma \ref{lemma:bounds}. So $S\leq \sum^{k}_{i=1}\beta_{2i}+1-A \leq I$, forcing $S=I$. Suppose $0<A<\sum^{k}_{i=1}\beta_{2i}$, then by Lemma \ref{l:ineq} we get,
$$x_n=\frac{\sum^{k}_{i=1}\beta_{2i}x_{n-2i}+x_{n-\ell}}{A+x_{n-\ell}}\geq \min\left(\frac{\sum^{k}_{i=1}\beta_{2i}x_{n-2i}}{A},1\right).$$
Thus, since we have assumed positive initial conditions, the solution is bounded below by some $L>0$ by Lemma \ref{lemma:bounds}. Let $I=\liminf x_{n}$ then either $I\geq \min(\frac{\sum^{k}_{i=1}\beta_{2i}}{A}I,1)$ forcing $I\geq 1$ or $\liminf x_{n}$ doesn't exist.  If $I\geq 1$, choose $0<\epsilon<1-\frac{A}{\sum^{k}_{i=1}\beta_{2i}}$ then for sufficiently large $n$,
$$x_n\geq\frac{\sum^{k}_{i=1}\beta_{2i}(I-\epsilon)+x_{n-\ell}}{A+x_{n-\ell}}> 1.$$
 In the case where $\liminf x_{n}$ doesn't exist clearly $x_{n}>1$ for sufficiently large $n$. So applying the change of variables in Equation \ref{eq:change1}, the result follows immediately from Theorem \ref{thm:trichotomy1}.
Suppose that $0<A=\sum^{k}_{i=1}\beta_{2i}$. In this case, we apply Lemma \ref{l:ineq} to get,
$$x_{n}=\frac{\sum^{k}_{i=1}\beta_{2i}x_{n-2i}+x_{n-\ell}}{A+x_{n-\ell}}\geq \min\left(\frac{\sum^{k}_{i=1}\beta_{2i}x_{n-2i}}{A},1\right)\geq  \min_{i=1,\dots, k}(x_{n-2i},1)\;\; n\in\mathbb{N}.$$ Thus, the solution is bounded below by the minimum of the number 1 and the initial conditions which were assumed to be positive. So, each solution has a lower bound $L>0$ which depends on the initial conditions. For a given solution with positive initial conditions the triangle inequality gives us,
$$|x_{n}-1|\leq\left|\frac{\sum^{k}_{i=1}\beta_{2i}(x_{n-2i}-1)}{\sum^{k}_{i=1}\beta_{2i}+L}\right|\leq \frac{\sum^{k}_{i=1}\beta_{2i}|x_{n-2i}-1|}{\sum^{k}_{i=1}\beta_{2i}+L}\;\; n\in\mathbb{N}.$$
Thus every solution with positive initial conditions converges to $1$.
\end{proof}
We finish this section by sketching a proof of the case where nonnegative initial conditions are allowed in Theorem \ref{thm:trichotomy3}.
\begin{thm} \label{thm:trichotomy4}
Consider the $k^{th}$ order rational difference equation,
\begin{equation} \label{eq:trichotomy4}
x_n=\frac{\alpha+\sum^{k}_{i=1}\beta_{2i}x_{n-2i}+x_{n-\ell}}{A+x_{n-\ell}},\quad n\in\mathbb{N}.
\end{equation}
Assume nonnegative parameters, nonnegative initial conditions so that the denominator is nonvanishing and that $\ell$ is odd. Under these
assumptions Equation \eqref{eq:trichotomy4} exhibits the following behavior.
\begin{enumerate}[i.]
\item When $A>\sum^{k}_{i=1}\beta_{2i}$, every solution converges to an equilibrium.
\item When $1+\sum^{k}_{i=1}\beta_{2i}\geq A+1 > \sum^{k}_{i=1}\beta_{2i}$ and $\alpha>0$, every solution converges to an equilibrium.
\item When $1+\sum^{k}_{i=1}\beta_{2i}\geq A+1 > \sum^{k}_{i=1}\beta_{2i}$ and $A=0$, every solution converges to an equilibrium.
\item When $1+\sum^{k}_{i=1}\beta_{2i}\geq A+1 > \sum^{k}_{i=1}\beta_{2i}$, $\alpha=0$ and $A>0$, every solution converges to a periodic solution of not necessarily prime period $gcd(I_{\beta})$, and there exist periodic solutions of prime period $gcd(I_{\beta})$.
\item When $A+1=\sum^{k}_{i=1}\beta_{2i}$, every solution converges to a periodic solution of not necessarily prime period $2gcd(I_{\beta})$, and there exist periodic solutions of prime period $2gcd(I_{\beta})$.
\item When $A+1 < \sum^{k}_{i=1}\beta_{2i}$, unbounded solutions exist for some choice of initial
conditions.
\end{enumerate}
\end{thm} 
\begin{proof}
If $\alpha>0$ or $A=0$ then $x_{n}>0$ for $n\geq k$, so the result follows from Theorem \ref{thm:trichotomy3}. If $\alpha=0$ and $A>0$, then we have the following recursive property. If $x_{n}>0$, then $x_{n+i}>0$ for all $i\in I_{\beta}$. Let $N_{f}$ be the Frobenius number of the set $\{\frac{i}{gcd(I_{\beta})}|i\in I_{\beta}\}$. Induction and the properties of the Frobenius number yield that if $x_{n}>0$, then $x_{n+gcd(I_{\beta})m}>0$ for all $m\geq N_{f}$. Thus each of the subsequences $\{x_{mgcd(I_{\beta})+a}\}^{\infty}_{m=1}$ are either identically zero or eventually positive. After decoupling the difference equation using Remark 1 in \cite{f2}, we may apply Theorem \ref{thm:trichotomy3} to all of the eventually positive sequences and this yields all parts of the result, except for the construction of a periodic solution of prime period $gcd(I_{\beta})$ in case iv. To construct such a periodic solution choose initial conditions $x_{-m}=\sum^{k}_{i=1}\beta_{i}-A$ if $-m\equiv 0 \mod gcd(I_{\beta})$ and $x_{-m}=0$ otherwise. After decoupling the difference equation using Remark 1 in \cite{f2}, we see that this is indeed a periodic solution since $\sum^{k}_{i=1}\beta_{i}-A$ and $0$ are both equilibria of the reduced equation in this case. 
\end{proof} 
\section{Conclusion}
The reader should keep in mind that several periodic trichotomy conjectures in \cite{cl}, which have not yet been established, do not fit into the three major families laid out in this article. When established, these special cases may be the prototypical examples for additional general families of periodic trichotomies. The reader should be careful with any attempt to generalize Theorem \ref{thm:trichotomy3} since equations with additional odd delays tend to exhibit chaos in a range of the parameters.
Finally, in case it is still unclear that Theorem \ref{thm:trichotomy3} covers new ground the reader should notice that Theorem \ref{thm:trichotomy3} gives a periodic trichotomy result for the previously unknown special case,
$$x_{n}=\frac{\alpha + \gamma x_{n-2}+\epsilon x_{n-4}+x_{n-7}}{A+x_{n-7}}.$$

\par
\end{document}